\documentclass[a4,12pt]{amsart}
\usepackage{amssymb,amstext,amsmath,amscd,amsthm,amsfonts,enumerate,latexsym}
\usepackage{color}
\usepackage[dvipdfmx]{graphicx}
\usepackage[all]{xy}
\theoremstyle{plain}
\newtheorem{thm}{Theorem}[section]

\newtheorem*{thm*}{Theorem}
\newtheorem*{cor*}{Corollary}

\newtheorem{prop}[thm]{Proposition}

\newtheorem{cor}[thm]{Corollary}

\newtheorem*{claim*}{Claim}

\theoremstyle{definition}
\newtheorem{defn}[thm]{Definition}

\newtheorem{ex}[thm]{Example}

\newtheorem{prob}[thm]{Problem}

\newtheorem{setting}[thm]{Setting}

\theoremstyle{remark}

\numberwithin{equation}{thm}

\def\pd{\operatorname{pd}}

\def\Ext{\operatorname{Ext}}

\def\Ker{\operatorname{Ker}}

\def\Hom{\operatorname{Hom}}
\def\RHom{\mathrm{{\bf R}Hom}}

\def\rank{\mathrm{rank}}

\def\m{\mathfrak m}
\def\n{\mathfrak n}

\newcommand{\rma}{\mathrm{a}}

\newcommand{\rme}{\mathrm{e}}

\newcommand{\rmr}{\mathrm{r}}

\newcommand{\rmI}{\mathrm{I}}

\newcommand{\rmK}{\mathrm{K}}

\newcommand{\rmQ}{\mathrm{Q}}

\newcommand{\calR}{\mathcal{R}}
\newcommand{\calS}{\mathcal{S}}
\newcommand{\calT}{\mathcal{T}}

\newcommand{\calX}{\mathcal{X}}

\newcommand{\fka}{\mathfrak{a}}

\newcommand{\fkc}{\mathfrak{c}}

\newcommand{\fkm}{\mathfrak{m}}

\newcommand{\fkq}{\mathfrak{q}}

\newcommand{\mapright}[1]{%
\smash{\mathop{%
\hbox to 1cm{\rightarrowfill}}\limits^{#1}}}

\newcommand{\mapleft}[1]{%
\smash{\mathop{%
\hbox to 1cm{\leftarrowfill}}\limits_{#1}}}

\def\gr{\mbox{\rm gr}}

\title{Topics on $2$-almost Gorenstein rings}

\author{Shiro Goto}
\address{Department of Mathematics, School of Science and Technology, Meiji University, 1-1-1 Higashi-mita, Tama-ku, Kawasaki 214-8571, Japan}
\email{shirogoto@gmail.com}

\author{Naoki Taniguchi}
\address{Department of Mathematics, School of Science and Technology, Meiji University, 1-1-1 Higashi-mita, Tama-ku, Kawasaki 214-8571, Japan}
\email{taniguchi@meiji.ac.jp}
\urladdr{http://www.isc.meiji.ac.jp/~taniguchi/}


\thanks{The first author was partially supported by JSPS Grant-in-Aid for Scientific Research (C) 16K05112. The second author was partially supported by JSPS Grant-in-Aid for Young Scientists (B) 17K14176.}

\begin{document}
\maketitle

\setlength{\baselineskip}{14.5pt}

\begin{abstract}
The notion of $2$-almost Gorenstein ring is a generalization of the notion of almost Gorenstein ring in terms of Sally modules of canonical ideals. In this paper, we deal with two different topics related to $2$-almost Gorenstein rings. The purposes are to determine all the Ulrich ideals in $2$-almost Gorenstein rings and to clarify the structure of minimal free resolutions of $2$-almost Gorenstein rings. 
\end{abstract}

%{\footnotesize \tableofcontents}

%%%%%%%%%%%%%%%%%%%%%%%%%%%%%%%%%%%%%%%%%%%%%%%%%%%%%%%%%%%%%%%%%%%%%%%%%%%%%%%%%%%%%%%%%%%%%%%%%%%%%%%%%%%%%%%%%%%%%%%%%%%%%%%%%%%%%%%%%%%%%%%%%%%%%%%%%%%%%%%%%%%%%%%%%%%%%%%%%%%%%%%%%%%%%%%%%%%%%%%%%%%%%%%%%%%%%%%%

\section{Introduction}

The goal of the series of researches \cite{CGKM, GGHV, GMP, GMTY1, GMTY2, GMTY3, GMTY4, GRTT, GTT, GTT2, T} is to find a new class of Cohen-Macaulay rings which is a natural generalization of Gorenstein rings in terms of homological algebra. Almost Gorenstein rings are one of the candidates for such a class of rings. Historically, the notion of almost Gorenstein ring in our sense originated from the work \cite{BF} of V. Barucci and R. Fr\"oberg in 1997, where they dealt with the notion for one-dimensional analytically unramified local rings. In \cite{GMP}, the first author, N. Matsuoka, and T. T. Phuong extended the notion to one-dimensional Cohen-Macaulay local rings, using the first Hilbert coefficients of canonical ideals. Furthermore, T. D. M. Chau, the first author, S. Kumashiro, and N. Matsuoka recently defined the notion of $2$-almost Gorenstein ring as a generalization of almost Gorenstein rings of dimension one.

To explain our motivation more precisely, let us review on the definition of $2$-almost Gorenstein rings. Throughout this paper, let $(R, \m)$ be a Cohen-Macaulay local ring with $\dim R=1$, possessing the canonical module $\rmK_R$. Let $I$ be a canonical ideal of $R$, that is, $I$ is an ideal of $R$, $I \neq R$, and $I \cong \rmK_R$ as an $R$-module. Assume that $I$ contains a parameter ideal $Q = (a)$ of $R$ as a reduction. We set $K = \frac{I}{a} = \{\frac{x}{a} \mid x \in I\}$ in the total ring $\rmQ(R)$ of fractions of $R$ and let $S =R[K]$. Therefore, $K$ is a fractional ideal of $R$ such that $R \subseteq K \subseteq \overline{R}$ and $S$ is a module-finite extension of $R$, where $\overline{R}$ stands for the integral closure of $R$ in $\rmQ(R)$. We denote by $\fkc = R:S$ the conductor. Let $\calT= \calR(Q)=R[Qt]$ and $\calR = \calR(I)=R[It]$ be the Rees algebras of $Q$ and $I$, respectively, where $t$ denotes an indeterminate over $R$. We set 
$$
\calS_Q(I) = I\calR/I\calT
$$
and call it the Sally module of $I$ with respect to $Q$ (\cite{V1}). Let $\rme_i(I)~(i=0, 1)$ be the $i$-th Hilbert coefficients of $R$ with respect to $I$. We then have $$\rank~\calS_Q(I) = \rme_1(I) - \left[\rme_0(I) - \ell_R(R/I)\right]$$ (\cite[Proposition 2.2 (3)]{GNO}).

With this notation, T. D. M. Chau, the first author, S. Kumashiro, and N. Matsuoka introduced the following. Note that the invariant $\rank~\calS_Q(I)$ is independent of the choice of canonical ideals $I$ and their minimal reductions $Q$ (\cite[Theorem 2.5]{CGKM}).

\begin{defn}(\cite[Definition 1.3]{CGKM})\label{1.1}
We say that $R$ is {\it a $2$-almost Gorenstein ring}, if $\rank~\calS_Q(I) =2$, that is, $\rme_1(I) = \rme_0(I) - \ell_R(R/I) + 2$.
\end{defn}

\noindent
It is known by \cite{GMP} that $R$ is a non-Gorenstein but almost Gorenstein ring if and only if $\rank~\calS_Q(I) =1$, or equivalently, $\rme_1(I) = \rme_0(I) - \ell_R(R/I) + 1$. Therefore, the notion of $2$-almost Gorenstein ring is one of the natural generalizations of almost Gorenstein rings.

In the present article, for the further study of $2$-almost Gorenstein rings, we investigate two different topics related to $2$-almost Gorenstein rings. The first one is to study the Ulrich ideals in $2$-almost Gorenstein rings. Remember that the notion of Ulrich ideal/module dates back to the work \cite{GOTWY} of the first author, K. Ozeki, R. Takahashi, K.-i. Watanabe, and K.-i. Yoshida in 2014, where they introduced the notion as a generalization of maximally generated maximal Cohen-Macaulay modules (\cite{BHU}) and started the basic theory. Typically, the maximal ideal of a Cohen-Macaulay local ring of minimal multiplicity is an Ulrich ideal and the higher syzygy modules of Ulrich ideals are Ulrich modules. In \cite{GOTWY, GOTWY2}, they determined all the Ulrich ideals of a Gorenstein local ring of finite CM-representation type and dimension at most $2$, using the techniques in the representation theory of maximal Cohen-Macaulay modules. Moreover, in \cite{GTT2}, the authors and R. Takahashi studied the structure of the complex $\RHom_R(R/I, R)$ for an Ulrich ideal $I$ in a Cohen-Macaulay local ring $R$. They  proved also that in a one-dimensional non-Gorenstein almost Gorenstein ring $R$, the only possible Ulrich ideal is the maximal ideal $\m$ of $R$ (\cite[Theorem 2.14 (1)]{GTT2}). In Section 2 of the present article, we shall study the natural question of how many Ulrich ideals are contained in a given $2$-almost Gorenstein ring. Our main result is stated as follows, where $\calX_R$ denotes the set of Ulrich ideals of $R$. Note that the $R/\fkc$-freeness of $K/R$ plays an important role for the analysis of $2$-almost Gorenstein rings (see \cite[Section 5]{CGKM}).

\begin{thm}\label{1.2}
Suppose that $R$ is a $2$-almost Gorenstein ring of minimal multiplicity. Then the following assertions hold true.
\begin{enumerate}
\item[$(1)$] If $K/R$ is a free $R/\fkc$-module, then $\calX_R =\{\fkc, \m\}$.
\item[$(2)$] If $K/R$ is not a free $R/\fkc$-module, then $\calX_R =\{\m\}$.
\end{enumerate}
\end{thm}

The other topic in this article is to give the necessary and sufficient condition for $R$ to be a $2$-almost Gorenstein ring in terms of their minimal free resolutions. Our results Theorem \ref{3.2} and Proposition \ref{3.4} have been strongly inspired by the previous work \cite[Theorem 7.8]{GTT} of the authors and R. Takahashi.

In what follows, unless otherwise specified, let $R$ be a Cohen-Macaulay local ring with maximal ideal $\fkm$. For each finitely generated $R$-module $M$, let $\mu_R(M)$ (resp. $\ell_R(M)$) denote the number of elements in a minimal system of generators of $M$ (resp. the length of $M$). We denote by $\mathrm{K}_R$ the canonical module of $R$.

%%%%%%%%%%%%%%%%%%%%%%%%%%%%%%%%%%%%%%%%%%%%%%%%%%%%%%%%%%%%%%%%%%%%%%%%%%%%%%%%%%%%%%%%%%%%%%%%%%%%%%%%%%%%%%%%%%%%%%%%%%%%%%%%%%%%%%%%%%%%%%%%%%%%%%%%%%%%%%%%%%%%%%%%%%%%%%%%%%%%%%%%%%%%%%%%%%%%%%%%%%%%%%%%%%%%%%%%%%%%%%%%%%%%%%%%%%%%%%%%%%%

\section{Ulrich ideals in $2$-almost Gorenstein rings}\label{section2}

Let $(R,\m)$ be a Cohen-Macaulay local ring with $\dim R=1$, possessing the canonical module $\rmK_R$. Let $I$ be an $\m$-primary ideal of $R$ and assume that $I$ contains a parameter ideal $Q=(a)$ of $R$ as a reduction.

\begin{defn}(\cite[Definition 1.1]{GOTWY})\label{2.1}
We say that $I$ is {\it an Ulrich ideal of $R$}, if 
\begin{enumerate}
\item $I \supsetneq Q$, $I^2=QI$, and
\item $I/I^2$ is a free $R/I$-module.
\end{enumerate}
\end{defn}

\noindent
Note that in Definition \ref{2.1}, Condition $(1)$ is equivalent to saying that the associated graded ring $\gr_I(R) = \bigoplus_{n\ge 0} I^n/I^{n+1}$ is a Cohen-Macaulay ring with $\rma(\gr_I(R))=0$, where $\rma(\gr_I(R))$ denotes the $\rma$-invariant of $\gr_I(R)$. Therefore,  Condition $(1)$ of Definition \ref{2.1} is independent of the choice of minimal reductions $Q$ of $I$. When $I=\m$, Condition $(2)$ is automatically satisfied and Condition $(1)$ is equivalent to saying that $R$ has minimal multiplicity which is not a regular local ring. Suppose  that $I^2=QI$ and consider the exact sequence
$$
0 \to Q/QI \to I/I^2 \to I/Q \to 0
$$
of $R$-modules. We then have that $I/I^2$ is a free $R/I$-module if and only if so is $I/Q$. When $\mu_R(I) =2$, the latter condition is equivalent to saying that $Q:I =I$, namely, $I$ is a good ideal in the sense of \cite{GIW} (see also \cite[Lemma 2.3, Corollary 2.6]{GOTWY}).

Let us begin with the following. See (\cite[Example 7.3]{GOTWY}) for the proof.

\begin{ex}[{\cite[Example 7.3]{GOTWY}}]\label{2.2}
Suppose that $I$ is an Ulrich ideal of $R$ such that $\mu_R(I)=2$.  Let us write $I=(a, b)$ with $b \in R$. Then $b^2 = ac$ for some $c \in I$, because $I^2=QI$. With this notation, the minimal free resolution of $I$ is given by 
$$
\Bbb F : \ \ \cdots \to R^2 \overset{
\begin{pmatrix}
-b & -c\\
a & b
\end{pmatrix}}{\longrightarrow}
R^2 \overset{
\begin{pmatrix}
-b & -c\\
a & b
\end{pmatrix}}{\longrightarrow}R^2 \overset{\begin{pmatrix}
a & b
\end{pmatrix}}{\longrightarrow} I \to 0,
$$
whence $\pd_R I = \infty$. By \cite[Theorem 2.8]{GTT2}, $I$ is a totally reflexive $R$-module, that is, $I$ is reflexive, $\Ext_R^p(I, R) =(0)$, and $\Ext_R^p(I^*, R) = (0)$ for every $p >0$, where $(-)^* = \Hom_R(-, R)$ denotes the $R$-dual functor.
\end{ex}

The purpose of this section is to explore the question of how many Ulrich ideals are contained in a given $2$-almost Gorenstein rings. To do this, let $K$ be a fractional ideal of $R$ such that $R \subseteq K \subseteq \overline{R}$ and $K \cong \rmK_R$ as an $R$-module, where $\overline{R}$ denotes the integral closure of $R$ in $\rmQ(R)$. We set $S =R[K]$ and $\fkc = R : S$. Then $\fkc = K : S$ (\cite[Lemma 3.5 (2)]{GMP}). It is proved by \cite[Theorem 1.4]{CGKM} that $R$ is a $2$-almost Gorenstein ring if and only if $\ell_R(R/\fkc)=2$. Therefore, if $R$ is a $2$-almost Gorenstein ring, there exists a minimal system $x_1, x_2, \ldots, x_n$ of generators of $\m$ such that $\fkc =(x_1^2) + (x_2, x_3, \ldots, x_n)$, where $n=\mu_R(\m)$ is the embedding dimension of $R$.  Let  $\rmr(R) =\ell_R(\Ext_R^1(R/\m, R))$ denote the Cohen-Macaulay type of $R$.

With this notation we have the following, which is the key in our argument.

\begin{prop}\label{2.3}
Suppose that $R$ is a $2$-almost Gorenstein ring. Let $I$ be an Ulrich ideal of $R$ such that $\mu_R(I)=2$. Then the following assertions hold true.
\begin{enumerate}
\item[$(1)$] $K/R$ is a free $R/\fkc$-module. 
\item[$(2)$] $\fkc = (x_2, x_3, \ldots, x_n)$.
\item[$(3)$] $I + \fkc = \m$.
\end{enumerate}
Hence, $\mu_R(\fkc) =n-1$ and $x_1^2 \in (x_2, x_3, \ldots, x_n)$.
\end{prop}

\begin{proof}
We have $K/R \cong (R/\fkc)^{\oplus \ell} \oplus (R/\m)^{\oplus m}$ as an $R/\fkc$-module for some $\ell >0$ and $m \ge 0$ such that $\ell +m = \rmr(R)+1$ (\cite[Proposition 3.3 (4)]{CGKM}). Let us assume that $m>0$. Then, since $I$ is totally reflexive (see Example \ref{2.2}) and $\Ext_R^p(I,K)= (0)$ for every $p > 0$, we get $\Ext_R^p(I, K/R)=(0)$, whence 
$$
\Ext_R^p(I, R/\m)=(0)
$$ for all $p>0$, because  $R/\m$ is a direct summand of $K/R$. This is impossible, since $I$ is not a free $R$-module.  Hence $m=0$ and Assertion $(1)$ follows.

Let us maintain the same notation (see Example \ref{2.2}). We then have a minimal free resolution   
$$
\Bbb F : \ \ \cdots \to R^2 \overset{
\begin{pmatrix}
-b & -c\\
a & b
\end{pmatrix}}{\longrightarrow}
R^2 \overset{
\begin{pmatrix}
-b & -c\\
a & b
\end{pmatrix}}{\longrightarrow}R^2 \overset{\begin{pmatrix}
a & b
\end{pmatrix}}{\longrightarrow} I \to 0
$$
of $I$. Let $\overline{x}$ denote, for each $x \in R$, the image of $x$ in $R/\fkc$. Then, by taking the $R/\fkc$-dual of the resolution $\Bbb F$, we get the exact sequence
$$ 
0 \to \Hom_R(I, R/\fkc) \to (R/\fkc)^{\oplus 2} \overset{
\begin{pmatrix}
-\overline{b} & \overline{a}\\
-\overline{c} & \overline{b}
\end{pmatrix}}{\longrightarrow}
(R/\fkc)^{\oplus 2} \overset{
\begin{pmatrix}
-\overline{b} & \overline{a}\\
-\overline{c} & \overline{b}
\end{pmatrix}}{\longrightarrow}
(R/\fkc)^{\oplus 2}
 \to \cdots 
$$
of $R$-modules, because $\Ext_R^p(I, R/\fkc) = (0)$ for all $p >0$ (remember that $R/\fkc$ is a direct summand of $K/R$). Hence $I \nsubseteq \fkc$. Therefore,  $I+\fkc =\m$, since $\ell_R(R/\fkc) = 2$. This proves Assertion $(3)$.

To see Assertion $(2)$, note that $\m/\fkc = (\overline{x_1}) = (\overline{a}, \overline{b})$ and $\ell_R(\m/\fkc)=1$. We set $J = (x_2, x_3, \ldots, x_n)$.

{\bf Case 1 ($\overline{a} \ne 0$).}  We write $a = \alpha x_1 + \xi$ for some $\alpha \in R$ and $\xi \in J$. We may assume $\alpha=1$ and $b \in J$, because $\alpha \notin \m$ and $\overline{b} \in \fkm/\fkc = (\overline{a})$. Hence 
$$
\begin{pmatrix}
-\overline{b} & \overline{a}\\
-\overline{c} & \overline{b}
\end{pmatrix}=
\begin{pmatrix}
0 & \overline{x_1}\\
-\overline{c} & 0
\end{pmatrix},
$$
so that $\overline{c} \ne 0$. Therefore, writing $c= \gamma x_1 + \rho$ with $\gamma \notin \m$ and $\rho \in J$, we get 
$$
\gamma x_1^2 \equiv ac = b^2 \equiv 0 \ \ \text{mod}\ \ J,
$$ 
whence $x_1^2 \in J$ and $\fkc =J$.

{\bf Case 2 ($\overline{a} = 0$).} We have $\m/\fkc = (\overline{b})$. We may assume $b = x_1 + \xi$ for some $\xi \in J$. Let $c= \alpha x_1 + \eta$ with $\alpha \in R$ and $\eta \in J$. Therefore, as $b^2 = ac$, we get 
$$
x_1^2 \equiv \beta x_1^2 {\cdot}\alpha x_1 \ \ \text{mod} \ \ J
$$
for some $\beta \in R$, where $a = \beta x_1^2 + \rho$ with $\beta \in R$ and $\rho \in J$. Hence $x_1^2 \in J$ and $\fkc = J$.
\end{proof}

\begin{cor}\label{1.7}
Suppose that $R$ is a $2$-almost Gorenstein ring and $S$ is a Gorenstein ring. If $I$ is an Ulrich ideal of $R$ with $\mu_R(I) =2$, then $\mu_R(\fkc) = n -1$ and $R$ doesn't have minimal multiplicity. 
\end{cor}

\begin{proof}
We have $\fkc = K : S$. Hence $\fkc \cong S$ as an $S$-module, because $S$ is a Gorenstein ring. Therefore, by Proposition \ref{2.3} and \cite[Proposition 3.3 (5)]{CGKM},  $n-1 = \mu_R(\fkc) = \mu_R(S) = \rmr(R) +1$. If $R$ has minimal multiplicity, then $\rmr(R) = n-1$, whence 
$$
n-1 = \mu_R(\fkc) = \rmr(R) +1 = n,
$$
which is a contradiction.
\end{proof}

The following Example \ref{2.7} ensures the existence of Ulrich ideals which are generated by two-elements.

\begin{ex}\label{2.7}
Let $A = R \ltimes R$ denote the idealization of $R$ over $R$. Let $\fkq$ be an arbitrary parameter ideal of $R$ and set $I = \fkq \times R$. Then by \cite[Example 2.2]{GOTWY}, $I$ is an Ulrich ideal of $A$ and $\mu_A(I)=2$. Note that if $R$ is a non-Gorenstein almost Gorenstein local ring, then $A$ is a $2$-almost Gorenstein ring (\cite[Theorem 3.9]{CGKM}).
\end{ex}

The $R/\fkc$-freeness of $K/R$ strongly influences the behavior of $2$-almost Gorenstein rings (see \cite[Section 5]{CGKM}). However, even if $K/R$ is not $R/\fkc$-free, $2$-almost Gorenstein rings still behave well. For example, we have the following, which shows that every $2$-almost Gorenstein ring $R$ is G-regular in the sense of \cite{greg}, namely, every totally reflexive $R$-module is free, provided $K/R$ is not $R/\fkc$-free.

\begin{thm}\label{1.5}
Suppose that $R$ is a $2$-almost Gorenstein ring and $K/R$ is not a free $R/\fkc$-module. Let $M$ be a finitely generated $R$-module such that $\Ext_R^p(M, R) =(0)$ for all $p \gg 0$. Then $\pd_RM \le 1$. 
\end{thm}

\begin{proof}
Consider the exact sequence
$$
0 \to L \to R^{\oplus s} \to M \to 0
$$
of $R$-modules, where $s = \mu_R(M)$. Then, for every $p \ge 2$, we get an isomorphism
$$
\Ext_R^{p-1}(L, R) \cong \Ext_R^p(M, R).
$$
Therefore, $\Ext_R^p(L, K/R) = (0)$ for all $p \gg 0$, because $\Ext_R^p(L,K) = (0)$ if $p \ge 1$,  whence for all $p \gg 0$
$$
\Ext_R^p(L, R/\m) =(0),
$$
because $R/\m$ is a direct summand of $K/R$. Thus $\pd_RL < \infty$, so that  $\pd_R M < \infty$. 
\end{proof}

Let us note the following, which answers the question of  when the conductor $\fkc=R:S$ is an Ulrich ideal of $R$.

\begin{prop}\label{1.6}
Suppose that $R$ is a $2$-almost Gorenstein ring. Then $\fkc$ is an Ulrich ideal of $R$ if and only if $S$ is a Gorenstein ring and $K/R$ is $R/\fkc$-free.
\end{prop}

\begin{proof}
Note that $S$ is a Gorenstein ring if and only if $\fkc^2 = f \fkc$ for some $f \in \fkc$ (\cite[Corollary 3.8]{GMP}). When this is the case, we have $\fkc=fS$, since the ideal $\fkc= K : S$ of $S$ is invertible. Therefore, to prove the assertion, we may assume that $\fkc = fS$ for some $f \in \fkc$. Then $\fkc/\fkc^2 \cong S/fS = S/\fkc$, so that $\fkc/\fkc^2$ is a free $R/\fkc$-module if and only if so is $S/\fkc$. The latter condition is equivalent to saying that $K/R$ is a free $R/\fkc$-module, because of  the exact sequence 
$$
0 \to R/\fkc \to S/\fkc \to S/R \to 0
$$
of $R$-modules and the direct sum decomposition $$S/R \cong K/R \oplus R/\fkc$$ of $S/R$ (\cite[Proposition 3.3]{CGKM}).
\end{proof}

We are now in a position to prove Theorem \ref{1.2}.

\begin{proof}[Proof of Theorem \ref{1.2}] Since $R$ has minimal multiplicity, $\m$ is an Ulrich ideal of $R$, whence $\calX_R \ne \emptyset$.

$(1)$~~ Suppose that $K/R$ is a free $R/\fkc$-module. By \cite[Proposition 5.7 (1)]{CGKM}, we have that $\m:\m$ is a local ring. Because the $2$-almost Gorenstein ring $R$ has minimal multiplicity, $S$ is a Gorenstein ring by \cite[Corollary 5.3]{CGKM}. Therefore, the conductor $\fkc$ is an Ulrich ideal of $R$ by Proposition \ref{1.6}, so that $\calX_R \supseteq \{\fkc, \m\}$. Let $I$ be an Ulrich ideal of $R$. Then $\mu_R(I) \ge 3$ by Corollary \ref{1.7}. Therefore, $\fkc = (0):_RK/R \subseteq I$ by \cite[Corollary 2.13]{GTT2}. Hence either $I = \fkc$ or $I = \m$, because $\fkc \subseteq I \subseteq \m$ and $\ell_R(R/\fkc)=2$.

$(2)$~~  Let $I$ be an Ulrich ideal of $R$. Then $\mu_R(I) \ge 3$ by Proposition \ref{2.3}, since $K/R$ is not $R/\fkc$-free. Therefore, by the proof of Assertion $(1)$,  either $I=\fkc$ or $I=\m$, so that $I = \fkm$, because $\fkc$ is not an Ulrich ideal of $R$ by Proposition \ref{1.6}.  
\end{proof}

Let us note a few examples.

\begin{ex}
Let $V=k[[t]]$ denote the formal power series ring over a field $k$ and set $R_1 =k[[t^3, t^7, t^8]]$ and $R_2=k[[t^4, t^9, t^{11}, t^{14}]]$. Let $K_i$ be a fractional ideal of $R_i$ such that $R_i \subseteq K_i \subseteq \overline{R_i}$ and $K_i \cong \rmK_{R_i}$ as an $R_i$-module. Then $R_1$ and $R_2$ are $2$-almost Gorenstein rings such that $K_1/R_1$ is a free $R/\fkc_1$-module but $K_2/R_2$ is not $R/\fkc_2$-free, where $\fkc_i = R_i : R_i[K_i]$. Hence $\calX_{R_1}=\{\fkc_1, (t^3, t^7, t^8)\}$ and $\calX_{R_2}=\{(t^4, t^9, t^{11}, t^{14})\}$.
\end{ex}

\begin{ex}
Let $V = k[[t]]$ be the formal power series ring over a field $k$ and set $R = k[[t^6,t^8,t^{10},t^{11}]]$. Then $R$ is a $2$-almost Gorenstein ring with $\rmr(R) = 2$ and $\fkc = (t^6,t^{8},t^{10})$. For this ring, $(t^6, t^{11}), (t^8, t^{11})$, and $\fkc$ are those Ulrich ideals generated by monomials in $t$.
\end{ex}

%%%%%%%%%%%%%%%%%%%%%%%%%%%%%%%%%%%%%%%%%%%%%%%%%%%%%%%%%%%%%%%%%%%%%%%%%%%%%%%%%%%%%%%%%%%%%%%%%%%%%%%%%%%%%%%%%%%%%%%%%%%%%%%%%%%%%%%%%%%%%%%%%%%%%%%%%%%%%%%%%%%%%%%%%%%%%%%%%%%%%%%%%%%%%%%%%%%

\section{The structure of a minimal free resolution of $2$-almost Gorenstein rings}

In this section we study the structure of minimal free resolutions of $2$-almost Gorenstein rings. For the rest of this article, we fix the following notation.

\begin{setting}\label{3.1}
Let $(T, \n)$ be a regular local ring with $\dim T= n$. Let $\fka$ be an ideal of $T$ such that $\fka \subseteq \n^2$. We set $R=T/\fka$ and $\m = \n/\fka$. Suppose that $R$ is a Cohen-Macaulay local ring of dimension one. Let $K$ be an $R$-submodule of $\rmQ(R)$. We assume that $R \subseteq K \subseteq \overline{R}$ and $K \cong \rmK_R$ as an $R$-module.

\end{setting}

First, suppose that $R$ is a $2$-almost Gorenstein ring and  choose a minimal system $x_1, x_2, \ldots, x_n$ of generators of $\m$ such that $\fkc = (x_1^2) + (x_2, x_3, \ldots, x_n)$ (this choice is possible; see \cite[Proposition 3.3]{CGKM}). Let $X_i \in \n$ such that $x_i=\overline{X_i}$ in $R$, where $\overline{X_i}$ denotes the image of $X_i$ in $R$. Hence $\n = (X_1, X_2, \ldots, X_n)$. We set $J = (X_1^2)+(X_2, X_3, \ldots, X_n)$; hence $\fkc = JR$. Then, because $\ell_T(T/J) = \ell_R(R/\fkc)=2$, we have  
$
T/J \cong R/\fkc
$, whence $\fka \subseteq J$. Remember that $K/R \cong (R/\fkc)^{\oplus \ell} \oplus (R/\m)^{\oplus m}$ as an $R$-module for some integers $\ell >0$ and $m \ge 0$ such that $\ell + m = \rmr(R)-1$ (\cite[Proposition 3.3]{CGKM}). Let us write $$K= R + \sum_{i=1}^\ell Rf_i + \sum_{j=1}^m R g_j$$
with $f_i, g_j \in K$ so that
$$
\sum_{i=1}^{\ell}(R/\fkc){\cdot} \overline{f_i} \cong (R/\fkc)^{\oplus \ell}\ \ \text{and} \ \  \sum_{j=1}^m (R/\fkc){\cdot}\overline{g_j} \cong (R/\m)^{\oplus m},
$$
where $f_i, g_j \in K$ and $\overline{f_i}, \overline{g_j}$ denotes their images in $K/R$. 

With this notation the first main result of this section is stated as follows. 

\begin{thm}\label{3.2}
Suppose that $R$ is a $2$-almost Gorenstein ring. Then the $T$-module $K$ has a  minimal free resolution of the form
$$
\cdots \to F_1 \overset{\Bbb M}{\longrightarrow} F_0 \overset{\Bbb N}{\longrightarrow} K \to 0,
$$
where
$$
\Bbb N= 
\begin{pmatrix}
-1 & f_1 f_2 \cdots f_{\ell} & g_1 g_2 \cdots g_m
\end{pmatrix} \ \  \text{and}
$$
$$\Bbb M = 
\begin{pmatrix}
a_{11} a_{12} \cdots a_{1n} & \cdots & a_{\ell1} a_{\ell2} \cdots a_{\ell n} & b_{11} b_{12} \cdots b_{1n} & \cdots & b_{m1} b_{m2} \cdots b_{mn} & c_1 c_2 \cdots c_q \\
X^2_1 X_2 \cdots X_n & 0 & 0 & 0  & 0 & 0  & 0 \\
0 & \ddots & 0 & 0  & 0 & 0  & 0 \\
\vdots & \vdots & X^2_1 X_2 \cdots X_n & \vdots & \vdots & \vdots & \vdots \\
0 & 0 & 0 & X_1 X_2 \cdots X_n & 0 & 0 & 0\\
0 & 0 & 0 & 0  & \ddots & 0 & 0 \\
0 & 0 & 0 & 0  & 0 & X_1 X_2 \cdots X_n & 0 
\end{pmatrix}
$$
such that $a_{ij} \in J$ for every $1 \le i \le \ell$, $1 \le j \le n$ and $b_{ij} \in J$ for every $1 \le i \le \ell$, $2 \le j \le n$.

We furthermore have
$$
\fka = \sum_{i=1}^{\ell} {\rm I}_2
\left(\begin{smallmatrix}
a_{i1} & a_{i2} & \cdots & a_{in} \\
X^2_1 & X_2 & \cdots & X_n
\end{smallmatrix}\right) + 
\sum_{i=1}^m {\rm I}_2
\left(\begin{smallmatrix}
b_{i1} & b_{i2} & \cdots & b_{in} \\
X_1 & X_2 & \cdots & X_n
\end{smallmatrix}\right) + (c_1, c_2, \ldots, c_q), 
$$ where $\rmI_2(\Bbb L)$ denotes, for each $2 \times n$ matrix $\Bbb L$ with entries in $T$, the ideal of $T$ generated by $2 \times 2$ minors of $\Bbb L$.
\end{thm}

\begin{proof}
Let 
$$
\cdots \longrightarrow F_1 \overset{\Bbb A}{\longrightarrow} F_0 \longrightarrow K \longrightarrow 0 \ \ \ \ (\sharp)
$$
be a minimal $T$-free resolution of $K$, where $F_0 = T \oplus T^{\oplus \ell} \oplus T^{\oplus m}$. 
Since $K/R \cong (T/J)^{\oplus \ell}\oplus (T/\n)^{\oplus m}$, $K/R$ has a minimal $T$-free resolution of the form
$$
\cdots \longrightarrow G_1 = T^{\oplus \ell n} \oplus T^{\oplus \ell m} \overset{\Bbb B}{\longrightarrow} G_0 = T^{\oplus \ell} \oplus T^{\oplus m} \longrightarrow K/R \longrightarrow 0
$$
where
$$\Bbb B = 
\begin{pmatrix}
X^2_1 X_2 \cdots X_n & 0 & 0 & 0  & 0 & 0   \\
0 & \ddots & 0 & 0  & 0 & 0   \\
\vdots & \vdots & X^2_1 X_2 \cdots X_n & \vdots & \vdots & \vdots  \\
0 & 0 & 0 & X_1 X_2 \cdots X_n & 0 & 0 \\
0 & 0 & 0 & 0  & \ddots & 0  \\
0 & 0 & 0 & 0  & 0 & X_1 X_2 \cdots X_n 
\end{pmatrix}.
$$
We consider the following presentation
$$
F_1 \overset{\Bbb A'}{\longrightarrow} G_0 \overset{\Bbb N'}{\longrightarrow} K/R \longrightarrow 0
$$
of $K/R$ induced from the above resolution  $(\sharp)$ of $K$, where $s=\rank_T F_1$, $\Bbb N'=
\begin{pmatrix}
\bar{f_1} \bar{f_2} \cdots \bar{f_{\ell}} & \bar{g_1} \bar{g_2} \cdots \bar{g_m}
\end{pmatrix}$, and $\Bbb A'$ is an $(\ell + m) \times s$ matrix obtained from $\Bbb A$ by deleting the first row. We then get the commutative diagram
$$
\xymatrix{
G_1 \ar[r]^{\Bbb B}\ar[d]^{\xi} & G_0 \ar[r]\ar[d]^\cong & K/R\ar[r]\ar[d]^\cong & 0\\
 F_1 \ar[r]^{\Bbb A'}\ar[d]^{\eta} & G_0 \ar[r]\ar[d]^\cong & K/R \ar[r] \ar[d]^\cong & 0 \\
G_1 \ar[r]^{\Bbb B} & G_0 \ar[r] & K/R \ar[r] & 0
}
$$
of $T$-modules such that $\eta \circ \xi$ is an isomorphism.
Hence $$
\Bbb A'\ Q= \left[ \ \Bbb B \mid  O \ \right]
$$
for some $s \times s$ invertible matrix $Q$ with entries in $T$, where $O$ denotes the null matrix. Therefore, setting $\Bbb M = \Bbb A Q$, we get
$$
\text{\large $\Bbb M$} =
\arraycolsep5pt
\left(
\begin{array}{@{\,}ccc@{\,}}
~ & * &~\\
\hline
&\multicolumn{1}{c}{\raisebox{-10pt}[0pt][0pt]{\large $\Bbb A'$}}\\
&&\\
\end{array}
\right)\cdot{\text{\large $Q$}} \ 
=
\arraycolsep5pt
\left(
\begin{array}{@{\,}c|c@{\,}}
~ * ~ & ~ * ~\\
\hline
\raisebox{-10pt}[0pt][0pt]{\large $\Bbb B$}&\raisebox{-10pt}[0pt][0pt]{\large $O$}\\
&\\
\end{array}
\right).
$$
Let us show that $a_{ij}, b_{ij} \in J$. We set $Z_1 =X_1^2$ and $Z_i = X_i$ for each $2 \le i \le n$. Then
$$
a_{ij}\cdot(-1) + Z_j\cdot f_i = 0
$$
for every $1 \le i \le \ell$ and $1 \le j \le m$, whence $a_{ij} \in J$, because $\fkc K \subseteq \fkc S = \fkc$ and $\fkc =J/\fka$. Since
$$
b_{ij}\cdot(-1) + Z_j\cdot g_i = 0,
$$
we get $b_{ij} \in J$, provided $j \ge 2$. The last assertion follows from the fact that $Z_1, Z_2, \ldots, Z_n$ is a $T$-regular sequence; see \cite[Proof of Theorem 7.8]{GTT} for detail.
\end{proof}

As a consequence of Theorem \ref{3.2} we have the following.

\begin{cor}
With the notation as in Theorem $\ref{3.2}$, we have the following.
\begin{enumerate}
\item[$(1)$] If $n = 3$, then $r = 2$, $q = 0$, $\ell =1$, and $m=0$, so that $\Bbb M = \left(\begin{smallmatrix}
a_{21} & a_{22} & a_{23} \\
X^2_1 & X_2 & X_3 
\end{smallmatrix}\right)$.
\item[$(2)$] If $R$ has minimal multiplicity, then $q = 0$.
\end{enumerate}
\end{cor}

\begin{proof}
(1) Suppose that $n = 3$. Then $R$ has a minimal $T$-free resolution
$$
0 \longrightarrow F_2 \overset{{}^t \Bbb M}{\longrightarrow} F_1 \longrightarrow F_0 \longrightarrow R \longrightarrow 0,
$$ 
where the matrix $\Bbb M$ has the form stated in Theorem \ref{3.2}. Therefore 
$$
\rmr(R)+1 = \rank_TF_1 = \ell n + mn + q = 3(\rmr(R)-1)+q,
$$
whence $4-2\cdot\rmr(R)=q \ge 0$. Thus $\rmr(R)=2$ and $q=0$. We get $\ell=1$, $m=0$, because $\ell + m = \rmr(R) -1$.

(2) Suppose that $R$ has minimal multiplicity. Then $\rmr(R)=n-1$. By \cite[{\sc Theorem} 1 (iii)]{S2} we have $\ell n + mn + q = (\rmr(R)-1)(\rmr(R)+1)$. Therefore, $\ell + m = \rmr(R) -1$, so that $q=0$.
\end{proof}

Let us give a sufficient condition for $R = T/\fka$ to be a $2$-almost Gorenstein ring in terms of its minimal free resolution. With the same notation as Setting \ref{3.1}, let us assume that $R$ is not a Gorenstein ring and that 
$$
K=R+ \sum_{i=1}^\ell Rf_i + \sum_{j=1}^m Rg_j
$$
with $f_i, g_j \in K$, where $\ell >0$ and $m \ge 0$ such that $\ell + m=\rmr(R)-1$. Then we have the following.

\begin{prop}\label{3.4}
Suppose that there is a regular system $X_1, X_2, \ldots, X_n$ of parameters of $T$ and that $K$ has a minimal $T$-free resolution of the form
$$
\cdots \to F_1 \overset{\Bbb M}{\longrightarrow} F_0 \overset{\Bbb N}{\longrightarrow} K \to 0  \quad \quad \quad (\sharp)
$$
where $\Bbb N= 
\begin{pmatrix}
-1 & f_1 f_2 \cdots f_{\ell} & g_1 g_2 \cdots g_m
\end{pmatrix}$ and 
$$\Bbb M = 
\begin{pmatrix}
a_{11} a_{12} \cdots a_{1n} & \cdots & a_{\ell1} a_{\ell2} \cdots a_{\ell n} & b_{11} b_{12} \cdots b_{1n} & \cdots & b_{m1} b_{m2} \cdots b_{mn} & c_1 c_2 \cdots c_q \\
X^2_1 X_2 \cdots X_n & 0 & 0 & 0  & 0 & 0  & 0 \\
0 & \ddots & 0 & 0  & 0 & 0  & 0 \\
\vdots & \vdots & X^2_1 X_2 \cdots X_n & \vdots & \vdots & \vdots & \vdots \\
0 & 0 & 0 & X_1 X_2 \cdots X_n & 0 & 0 & 0\\
0 & 0 & 0 & 0  & \ddots & 0 & 0 \\
0 & 0 & 0 & 0  & 0 & X_1 X_2 \cdots X_n & 0 
\end{pmatrix}
$$
with $q \ge 0$, $a_{ij} \in (X_1^2) + (X_2, X_3, \ldots, X_n)$ for every $1 \le i \le \ell$, $1 \le j \le n$, and $b_{ij} \in (X_1^2) + (X_2, X_3, \ldots, X_n)$ for every $1 \le i \le \ell$, $2 \le j \le n$. Then $R$ is a $2$-almost Gorenstein ring.
\end{prop}

\begin{proof}
The above resolution $(\sharp)$ induces a presentation 
$$
F_1 \overset{\Bbb B}{\longrightarrow} G_0 \overset{\Bbb L}{\longrightarrow} K/R \longrightarrow 0
$$
of $K/R$, where 
$\Bbb L=
\begin{pmatrix}
\bar{f_1} \bar{f_2} \cdots \bar{f_{\ell}} & \bar{g_1} \bar{g_2} \cdots \bar{g_m}
\end{pmatrix}$ and
$$\Bbb B = 
\begin{pmatrix}
X^2_1 X_2 \cdots X_n & 0 & 0 & 0  & 0 & 0  \\
0 & \ddots & 0 & 0  & 0 & 0  \\
\vdots & \vdots & X^2_1 X_2 \cdots X_n & \vdots & \vdots & \vdots \\
0 & 0 & 0 & X_1 X_2 \cdots X_n & 0 & 0 \\
0 & 0 & 0 & 0  & \ddots & 0 \\
0 & 0 & 0 & 0  & 0 & X_1 X_2 \cdots X_n 
\end{pmatrix}
$$
(here $\overline{f_i}$ and $\overline{g_j}$ denote, respectively, the images of $f_i$ and $g_j$ in $K/R$). 
Hence $$K/R \cong (T/J)^{\oplus \ell} \oplus (T/\n)^{\oplus m}$$ as a $T$-module, so that $\n{\cdot}(K/R) \ne (0)$, because $\ell > 0$. Consequently, $\fkc \subsetneq \m$. We set  $J = (X_1^2)+(X_2, X_3, \ldots, X_n)$ and $I = JR$. Note that $X_1^2{\cdot}f_i = \overline{a_{i1}}$ and $X_k{\cdot}f_i = \overline{a_{ik}}$ for all $1 \le i \le \ell$ and $2 \le k \le n$, because $a_{ik} \in J$, where $\overline{a}$ denotes, for each $a \in T$, the image of $a$ in $R$. We similarly have $X_k{\cdot}g_j \in I$ for all $1 \le j \le m$ and $1 \le k \le n$, because $b_{jk} \in J$. Thus $IK \subseteq I$. Consequently, $IS \subseteq I$, since $S = K^q$ for all $q \gg 0$. Hence $I \subseteq \fkc \subsetneq \m$, so that  $I = \fkc$, because $\ell_R(R/I) \le 2$. Therefore, $\ell_R(R/\fkc) = 2$, and hence by \cite[Theorem 1.4]{CGKM}, $R$ is a $2$-almost Gorenstein ring.
\end{proof}

In general, we cannot expect that $q=0$ in Theorem \ref{3.2}. Let us give one example.

\begin{ex}\label{3.5}
Let $V = k[[t]]$ be the formal power series ring over a field $k$ and set $R = k[[t^5,t^7,t^9,t^{13}]]$. Then $R$ is a $2$-almost Gorenstein ring with $\rmr(R)= 2$, $\fkc= (t^{10},t^7,t^9,t^{13})$, and $K=R + Rt^3$  (\cite[Example 5.5]{CGKM}). Let $T=k[[X,Y,Z,W]]$ be the formal power series ring and let $\varphi : T \to R$ be the $k$-algebra map defined by $\varphi (X) = t^5, \varphi(Y)=t^7, \varphi(Z) = t^9$, and $\varphi(W) = t^{13}$. Then $R$ has a minimal $T$-free resolution
of the form
$$0 \to T^2 \overset{\Bbb M}{\to} T^6 \to T^5 \to T \to R \to 0$$
where $${}^t\Bbb M = 
\begin{pmatrix}
W & X^2 & XY & YZ & Y^2 - XZ & Z^2 - XW \\
X^2 & Y & Z & W & 0 & 0\\
\end{pmatrix}.$$  We have $$\Ker \varphi = \rmI_2
\left(\begin{smallmatrix}
W & X^2 & XY & YZ \\
X^2 & Y & Z & W \\
\end{smallmatrix}\right) + (Y^2 - XZ, Z^2 - XW)
$$
and $K/R$ is a free $R/\fkc$-module.
\end{ex}

We note one more example.

\begin{ex}\label{3.6}
Let $V = k[[t]]$ be the formal power series ring over a field $k$ and set $R = k[[t^4,t^9,t^{11},t^{14}]]$. Then $R$ is a $2$-almost Gorenstein ring with $\rmr(R)=3$, $\fkc = (t^8,t^9,t^{11},t^{14})$, and $K = R+Rt^3+Rt^5$ (\cite[Example 5.6]{CGKM})). We consider the $k$-algebra map $\varphi : T \to R$ defined by $\varphi (X) = t^4, \varphi(Y)=t^9, \varphi(Z) = t^{11}$, and $\varphi(W) = t^{14}$, where $T=k[[X,Y,Z,W]]$ denotes the formal power series ring. Then $R$ has a minimal $T$-free resolution $$0 \to T^3 \overset{\Bbb M}{\to} T^8 \to T^6 \to T \to R \to 0$$ where 
$$
{}^t\Bbb M = 
\begin{pmatrix}
-Z & -X^3 & -W & -X^2Y & Y & W & X^4 & X^2Z \\
X^2 & Y & Z & W & 0 & 0 & 0 & 0\\
0 & 0 & 0 & 0 & X & Y & Z & W
\end{pmatrix}.
$$
We have
$$
\Ker \varphi = \rmI_2
\left(\begin{smallmatrix}
-Z & -X^3 & -W & -X^2Y \\
X^2 & Y & Z & W \\
\end{smallmatrix}\right) + 
\rmI_2
\left(\begin{smallmatrix}
Y & W & X^4 & X^2Z \\
X & Y & Z & W \\
\end{smallmatrix}\right)
$$
and $K/R$ is not a free $R/\fkc$-module.
\end{ex}

%%%%%%%%%%%%%%%%%%%%%%%%%%%%%%%%%%%

%%%%%%%%%%%%%%%%%%%%%%%%%%%%%%%%%%%%%%%

%\addcontentsline{toc}{section}{references}

\end{document}